\documentclass{aptpub}

\usepackage{amssymb}
\usepackage{graphicx}
\usepackage{url}

\authornames{E. Yudovina} 
\shorttitle{Simple model of limit order book} 

\newcommand{\BZ}{{\mathbb Z}}
\newcommand{\BR}{{\mathbb R}}
\newcommand{\BE}{{\mathbb E}}
\newcommand{\BP}{{\mathbb P}}
\newcommand{\CE}{{\mathcal E}}

\newcommand{\CL}{{\mathcal L}}

\newcommand{\one}{{\mathbf 1}}

\newcommand{\ess}{\text{ess\,}}
\providecommand{\abs}[1]{\left\lvert#1\right\rvert}
\providecommand{\norm}[1]{\left\lVert#1\right\rVert}

\numberwithin{equation}{section}  

\begin{document}

\title{A simple model of a limit order book} 

\authorone[University of Michigan]{Elena Yudovina} 

\addressone{Department of Statistics\\ 439 West Hall\\ 1085 South University Ave.\\ Ann Arbor, MI 48109\\
\email{yudovina@umich.edu}} 

\begin{abstract}
We formulate a simplified model of a limit order book, in which the arrival process is independent of the current state. We prove a phase transition result: there exist prices $\kappa_b$ and $\kappa_a$ such that, for any $\epsilon > 0$, only finitely many bid (ask) departures occur at prices below $\kappa_b-\epsilon$ (above $\kappa_a+\epsilon$), while the interval $(\kappa_b+\epsilon, \kappa_a - \epsilon)$ infinitely often contains no bids, and infinitely often contains no asks. We derive expressions for $\kappa_b$ and $\kappa_a$, which we solve in the case of uniform arrivals. We conjecture the positive recurrence of a modified model, and find the steady-state distribution of the highest bid and of the lowest ask assuming the positive recurrence.
\end{abstract}


\section{Introduction}
A limit order book is a pricing mechanism for a single-commodity market, in which users can trade off time against price by submitting orders to be executed at a later time, once the price becomes acceptable. This mechanism is used in many financial markets, and has generated extensive research, both empirical and theoretical. We do not aim to give an overview of the field here; references can be found in the survey by Gould et al. \cite{Gould}.

While much of the research has been either empirical studies of real-world markets, or game-theoretic analysis, our approach is to consider a Markovian model. This avoids the difficulties of prescribing models of individual user behaviour by assuming certain stochastic dynamics for the market as a whole. The pioneering paper of Gode and Sunder \cite{Gode_Sunder} showed that many of the features of a market may be reproduced even with zero-intelligence traders. Our model is somewhat similar to the models considered by Cont and de Larrard \cite{Cont_deLarrard}, Cont et al. \cite{Cont_et_al}, and Simatos \cite{Simatos}; however, the set-up differs from their work because we model the arrival events as independent of the state of the system. This assumption can be interpreted as treating the system on relatively short time scales, where the price does not significantly change. We discuss this at greater length in Remark~\ref{rem: assumptions}, after formulating our model.

It is surprising that even in such a simple setting, nontrivial behaviour emerges. Specifically, we find that the system experiences a phase transition: at prices below a certain threshold, only finitely many bid orders will ever be executed; at prices above the threshold, all the queued orders will clear infinitely often. (And similarly for ask orders, of course.) The probabilistic techniques used in this paper are not difficult, and showcase the fact that our model is attractive and amenable to analysis; we outline some of the extensions that could be considered in Section~\ref{section: future}.

\subsection{Basic notation}
For a process $\xi_t$ indexed by time, $\xi_{t-}$ indicates the state of $\xi$ just before $t$. We will take all our processes to be right-continuous.

For a set $A$, $\one_A$ is the indicator function of $A$. For two sets $A$ and $B$, $A \Delta B$ is the symmetric difference $(A \setminus B) \cup (B \setminus A)$.

\section{Model}
The basic dynamics of the system are as follows. At time 0, the limit order book is empty. Limit bids and asks arrivals form two independent point processes in $\BR \times \BR_+$. Arriving orders are iid; in particular, the interarrival times, types, and prices of the arriving orders are independent of each other and of the state of the limit order book. For convenience, we will assume that the distributions of prices of arriving orders are absolutely continuous, and write $f_b$, respectively $f_a$, for the density of the price of the arriving bid, respectively ask. We write $F_b$ and $F_a$ for the cumulative density. For the interarrival times, we assume that the event ``infinitely many orders arrive; no two orders arrive at the same time; finitely many orders arrive over any finite time interval'' has probability 1. In all that follows, we work on the probability-1 event that all order arrival times and prices are distinct.

The state of the limit order book at time $t$ is the two counting measures of bid orders and ask orders present in the limit order book. (Under our assumption, this is simply the set of prices of bids and asks.) Additionally, we keep track of the highest bid price $\beta_t$ and lowest ask price $\alpha_t$ inside the limit order book. We define $\beta_t = -\infty$ if there are no bids inside the book, and similarly $\alpha_t = \infty$ if there are no asks.

The change to the limit order book that occurs upon arrival depends on the location of the price of the arriving order relative to these two prices. If the arriving order at time $t$ is an ask at price $p$, then:\\
(i) If $p < \beta_{t-}$, the newly-arrived ask causes the bid at price $\beta_{t-}$ to be executed and leave. In this case, $\beta_t < \beta_{t-}$, and the ask side is unchanged.
(ii) If $\beta_{t-} < p < \alpha_{t-}$, the newly arrived ask joins the limit order book, and $\alpha_t = p$.
(iii) If $\alpha_{t-} < p$, the newly arrived ask joins the limit order book, and $\alpha_t = \alpha_{t-}$.

Similarly, if the arriving order at time $t$ is a bid at price $q$, then:\\
(i) If $q > \alpha_{t-}$, the newly-arrived bid causes the ask at price $\alpha_{t-}$ to be executed and leave. In this case, $\alpha_t > \alpha_{t-}$, and the bid side is unchanged.
(ii) If $\alpha_{t-} > q > \beta_{t-}$, the newly arrived bid joins the limit order book, and $\beta_t = q$.
(iii) If $\beta_{t-} > q$, the newly arrived bid joins the limit order book, and $\beta_t = \beta_{t-}$.

Abandonments are not allowed. Thus, bids and asks may depart if they are executed, or remain in the limit order book forever.

It follows from these dynamics that $\beta_t < \alpha_t$ always, i.e. all the bid orders in the limit order book are to the left of all the ask orders.

\begin{remark}
A convenient way to interpret bids that arrive at prices above $\alpha_t$ is to think of all of them as \emph{market bid} orders arriving at the current best price $\alpha_t$, and similarly for asks arriving at prices below $\beta_t$. The rate at which market bid orders arrive will depend on the current lowest ask price $\alpha_t$; it will increase when $\alpha_t$ is low, and decrease when $\alpha_t$ is high (and similarly for asks). In particular, the total rate at which market orders arrive will be higher when the bid-ask spread $\alpha_t - \beta_t$ is low, and lower when it is high.
\end{remark}

\begin{remark}\label{rem: assumptions}
This model differs from a real limit order book in several important aspects, which we now discuss.

First, ignoring abandonments represents a great difference to real-world markets, where a large fraction of the orders are canceled before being executed. However, if we consider the model only over relatively short time scales, then orders that are eventually canceled may be treated as ``remaining in the system forever'', while orders that are canceled very quickly may be interpreted as background noise. Properly incorporating abandonments into our model would be difficult, because the results of Section~\ref{section: monotonicity} rely on only the best orders departing, and then doing so in bid-ask pairs.

Second, we consider the arrival process of the orders to be independent from the state of the limit order book. Allowing the arrivals to depend on the state of the book would require a better understanding of the book's shape; this is work in progress. However, over relatively short time scales, during which the price does not shift substantially, we may expect this assumption to be reasonable.

Last, we consider the arriving orders to all have size 1. Allowing non-unit-sized orders also requires a better understanding of the dynamics of the shape of the limit order book, because a large arriving order may substantially move the highest price. In particular, in binned models (defined below), a large arriving order may remove orders from several bins, as opposed to just one. Our analysis can be extended without substantial change to accommodate orders of the form ``buy $n$ units or all of the orders available at the best market price, but do not buy any orders at higher prices''.
\end{remark}

\subsection{Modifications}
We will consider a variant of the model with a finite number of price ticks, or bins. We partition $\BR$ into some number (possibly infinite) of disjoint convex nonempty subsets (i.e. points or intervals). We will consider two versions of binned models:\\
(a) Ordinary binned limit order book: the arriving bid at price $p$ is allowed to depart with the lowest ask if $p$ and $\alpha_{t-}$ fall into the same bin (even if $p < \alpha_{t-}$), or if $p > \alpha_t$, and similarly for arriving asks.\\
(a) Strict binned limit order book: the arriving bid at price $p$ is allowed to depart with the lowest ask only if $p > \alpha_{t-}$ and they fall into different bins, and similarly for arriving asks.

In the binned models, we can treat the arrival price distribution as being supported on the set of bins. However, for coupling arguments it will be convenient to think of arrivals as coming from an underlying continuous distribution on $\BR$.

Additionally, we may consider non-zero initial states of the limit order book. In particular, we may allow the initial state to have an infinite number of bids or asks at a certain price, as long as always all bids are lower than all asks. This has the interpretation of a large player in the market, who is offering infinite liquidity at some (low) buy price, and some (high) ask price. Note that if we have infinitely many bids at some price $p$, then bids below $p$ will never leave the system, while arriving asks below $p$ will always immediately depart. Thus, we may focus our attention only on the prices above $p$.

\begin{remark}[Coordinate transformation]
It will sometimes be convenient for us to change coordinates so that the bids and asks arrive on $[0,1]$, and, moreover, the bid distribution $f_b$ is uniform over $[0,1]$. This can be done, e.g., by applying the transformation $x \mapsto F_b(x)$.
\end{remark}

\section{Results}
We now state our main results.

\begin{thm}\label{thm:kappa}
For any of the variants of the limit order book discussed above, there exist deterministic constants $\kappa_b$ and $\kappa_a$ with the following properties. For any $\epsilon > 0$,
\begin{itemize}
\item $\beta_t < \kappa_b - \epsilon$ occurs only finitely many times; $\beta_t < \kappa_b + \epsilon$ occurs infinitely often. Thus, bids below $\kappa_b - \epsilon$ eventually never leave, while above $\kappa_b + \epsilon$ infinitely often there are no bids.
\item Similarly, $\alpha_t > \kappa_a + \epsilon$ occurs only finitely many times; $\alpha_t > \kappa_a - \epsilon$ occurs infinitely often. Thus, asks above $\kappa_a + \epsilon$ eventually never leave, while below $\kappa_a - \epsilon$ infinitely often there are no asks.
\end{itemize}
\end{thm}
This indicates a sharp phase transition in the behaviour of the orders at low, medium, and high prices. We will identify the threshold values $\kappa_b$ and $\kappa_a$ below.

The following alternative characterization of $\kappa_b$ and $\kappa_a$ will be useful. For a limit order book $\CL$, let $B_t(p)$ denote the number of bids at time $t$ at prices $\leq p$, and let $A_t(p)$ denote the number of asks at time $t$ at prices $\geq p$. Note that asks are counted from the right. Clearly, we have $B_t(\infty) = B_t(\beta_t)$ and similarly $A_t(-\infty) = A_t(\alpha_t)$.
\begin{cor}\label{coroll: liminf}
Suppose the arrival process is Poisson of rate 1 in time, and the arrival price distributions are continuous. The values of $\kappa_b$ and $\kappa_a$ may be found as
\[
F_b(\kappa_b) = \liminf_{T \to \infty} \frac1T B_T(\infty)
\]
and
\[
1-F_a(\kappa_a) = \liminf_{T \to \infty} \frac1T A_T(-\infty).
\]
\end{cor}
The proof of these results appears in Section~\ref{section:coupling}.

The surprising fact is that we can obtain numeric values of $\kappa_b$ and $\kappa_a$ in terms of the arrival distributions.
\begin{thm}\label{thm:main}
Suppose arriving orders are equally likely to be bids and asks, and the densities $f_b$ and $f_a$ are absolutely continuous with respect to each other. Suppose further that $\kappa_b$ and $\kappa_a$ are known to be finite; for example, this is the case if there exist $x<y$ with the property that $F_b(x)=1-F_a(y)$, $F_b(y)=1-F_a(x)$, and $F_b(y)>1-2F_b(x)(1-F_b(x))$. (See Lemma~\ref{lm:finite kappa})

Then the threshold values $\kappa_b$ and $\kappa_a$ are the unique pair of finite numbers satisfying $F_b(\kappa_b) = 1-F_a(\kappa_a)$, such that the solution $\varpi^b$ of the second-order ODE
\[
\left(-\frac{f_a(x)}{1-F_b(x)} \left(F_a(x) \varpi^b(x) \right)' \right)' = \varpi^b(x) f_b(x)
\]
with initial conditions
\[
\varpi^b(\kappa_b) = \frac{1}{F_a(\kappa_b)},~~ \frac{d}{dx}\varpi^b(x)\vert_{x=\kappa_b} = -\frac{f_a(\kappa_b)}{F_a(\kappa_b)^2}
\]
satisfies $\varpi^b(x) \to 0$ as $x \uparrow \kappa_a$.

If $f_b = f_a = \one_{[0,1]}$, then
\[
\varpi^b(x) = (1-\kappa) \left(\frac1x + \log\left(\frac{1-x}{x}\right) \right), ~~ x \in (\kappa,1-\kappa)
\]
where if $w$ is the unique solution to $w e^w = e^{-1}$, then $\kappa = \frac{w}{w+1} \approx 0.217$.
\end{thm}
The proof of this result is in Section~\ref{section:proof of main}.

We conjecture that $\varpi^b(x) f_b(x)$ is the steady-state density of the distribution of $\beta_t$ (with respect to the Lebesgue measure). Unfortunately, we have been unable to show the positive recurrence that would imply the existence of a steady-state density for $\beta_t$, so instead in Lemma~\ref{lm: pi eqns} and proof of Theorem~\ref{thm:main} we derive that this quantity is the ergodic limit of the empirical distribution of $\beta_t$ along a certain sequence of times. We conjecture that the true result is as follows.

\begin{conjecture}\label{conj:recurrent}
Let the arrival process be as in Theorem~\ref{thm:main}.
\begin{enumerate}
\item Consider a binned limit order book with infinitely many bids in the bin containing $\kappa_b$, and infinitely many asks in the bin containing $\kappa_a$. (Its state is described by the number of and type of orders in the bins between these two.) This limit order book is recurrent.

\item Let $\epsilon > 0$ be fixed. Consider a limit order book $\CL$ whose initial state has infinitely many bids at $\kappa_b+\epsilon$ and infinitely many asks at $\kappa_a-\epsilon$. (If $\kappa_b$ and $\kappa_a$ are infinite, put the bids and asks at $F_b^{-1}(\epsilon)$ and $F_a^{-1}(1-\epsilon)$.) This limit order book is positive Harris recurrent.
\end{enumerate}
\end{conjecture}

Our analysis of $\varpi^i$, $i=a,b$ in Lemma~\ref{lm: pi eqns} and the proof of Theorem~\ref{thm:main} will show the following corollary.
\begin{cor}
Suppose Conjecture~\ref{conj:recurrent} holds. Let $\varpi^b_\epsilon$ and $\varpi^a_\epsilon$ be the distribution of the rightmost bid and the leftmost ask in the limit order books with infinitely many bids at $\kappa_b+\epsilon$ and asks at $\kappa_a - \epsilon$. As $\epsilon \to 0$, we have $\varpi^b_\epsilon(x) \to \varpi^b(x) f_b(x)$ and $\varpi^a(x) \to \varpi^a(x) f_a(x)$ uniformly.
\end{cor}
Evidence (theoretical and numerical) supporting the conjecture is presented in Section~\ref{section:recurrence}.

\section{Coupling and monotonicity}\label{section:coupling}
We now present some coupling arguments, which show monotonicity properties of our system. Our results will compare behaviors of limit order books $\CL$ and $\tilde \CL$ with the same underlying arrival process; we will consider the effect of changing the initial state and the effect of changing the bid-ask matching rule by changing the binning. We will refer to the state of the limit order books at time $t$ as $L_t$ and $\tilde L_t$ respectively.

\subsection{Initial state}

\begin{lem}
Let $\CL$ and $\tilde \CL$ be two limit order books with the same arrival process and order matching rule (i.e., both ordinary, or both binned with the same bins, or both strict binned with the same bins). Suppose the initial state $\tilde L_0$ differs from $L_0$ by the addition of a single bid. Then at all times $t$, the state $\tilde L_t$ differs from $L_t$ either by the addition of a single bid, or by the removal of a single ask. Similarly, if $\tilde L_0$ differs from $L_0$ by the addition of a single ask, then $\tilde L_t$ differs from $L_t$ either by the addition of a single ask or removal of a single bid.
\end{lem}

\begin{proof}
We prove the statement for the case of an extra bid, the case of an extra ask being entirely similar. The proof proceeds by induction on the number of arriving orders.

Clearly the statement is true before any orders arrive. Moreover, until the extra bid is removed in $\tilde \CL$, the order arrivals and departures in $\CL$ and $\tilde \CL$ coincide. Consider the time when the extra bid is removed in $\tilde \CL$; this corresponds to the arrival of an ask at some price $p$. Now, if in $\CL$ this ask also immediately departs (with some other bid at price $q$), then the state of $\tilde \CL$ differs from the state of $\CL$ by the addition of a bid (at price $q$). If, however, in $\CL$ the ask does not immediately depart, then the state of $\tilde \CL$ differs from the state of $\CL$ by the removal of this ask.
\end{proof}

We obtain some easy, but useful corollaries.
\begin{cor}\label{coroll:bounded perturbations}
Consider a limit order book $\CL$, and construct $\tilde \CL$ by, at some finite number of points in time, adding or removing a finite number of orders from $\CL$, for a total of at most $M$. Then at all times, the states of $\CL$ and $\tilde \CL$ differ by at most $M$ orders.
\end{cor}

\begin{cor}\label{coroll:monotonicity}
Consider a limit order book $\CL$, and construct $\tilde \CL$ by, at some points in time, adding some number of bids (but leaving the asks unchanged). Then at all times $\tilde \CL$ contains all of the bids in $\CL$ (and possibly some more), and a subset of the asks. Similarly, if we add some number of asks, but leave the bids unchanged, then $\tilde L$ will contain all of the asks in $\CL$, and a subset of the bids.
\end{cor}

We are now in a position to prove Theorem~\ref{thm:kappa}.
\begin{proof}[Proof of Theorem~\ref{thm:kappa}]
Our goal is to show that the event
\[
\CE(x) = \{\beta_t < x \text{ infinitely often}\}
\]
occurs with probability 0 or 1 for any $x$. Note that for $x < y$ we have $\CE(x) \subset \CE(y)$; we will take $\kappa_b = \inf\{x: \BP(\CE(x)) = 1\}$.

For $M \geq 0$, let
\[
\CE^M(x) = \{\beta_t < x \text{ infinitely often in $\tilde \CL$}\},
\]
where $\tilde \CL$ is the limit order book whose initial state is the same as that of $\CL$, but the first $M$ arrivals do not happen. We will show that $\CE(x) = \CE^M(x)$ with probability 1. Thus, $\bigcap\limits_M \CE^M(x)$ is a tail event which coincides with $\CE(x)$ with probability 1. By Kolmogorov's 0-1 law, $\BP(\bigcap\limits_M \CE^M(x)) \in \{0,1\}$, which proves the result.

We now show $\CE(x) = \CE^M(x)$ almost surely. By Corollary~\ref{coroll:bounded perturbations}, along every trajectory the states of $\CL$ and $\tilde \CL$ differ by at most $M$ orders. In particular,
\[
\beta_t \leq x \implies \tilde B_t(x) \leq M,
\]
and conversely, $\tilde \beta_t \leq x \implies B_t(x) \leq M$. (Recall $B_t(x)$ counts the number of bids at time $t$ at prices $\leq x$.)

Clearly, for $x < \ess\inf(f_a)$, neither $\CE$ nor $\CE^M$ occur, since no bid departures can happen. Therefore, let $x > \ess\inf(f_a)$. Whenever $B_t(x) \leq M$, the conditional probability that $M$ order arrivals later we will have
\[
\beta_{(t + \text{$M$ interarrival times})} < x
\]
is bounded below uniformly in $t$; and similarly for $\tilde B$ and $\tilde \beta$. Consequently, $\BP(\CE(x) \Delta \CE^M(x)) = 0$ as required.

The proof for $\kappa_a$ is entirely similar.
\end{proof}

We now prove Corollary~\ref{coroll: liminf}.
\begin{proof}[Proof of Corollary~\ref{coroll: liminf}]
Pick $\epsilon > 0$. By the definition of $\kappa_b$, and the strong law of large numbers for the arrival process, we have
\[
\lim_{T \to \infty} \frac1T B_T(\kappa_b - \epsilon) = F_b(\kappa_b - \epsilon).
\]
Since the number of bids elsewhere in the book is nonnegative, we obtain
\[
\liminf_{T \to \infty} \frac1T B_T(\infty) \geq F_b(\kappa_b - \epsilon).
\]
Moreover, we know that there exists a sequence of times $T_n \to \infty$ along which $\beta_{T_n} < \kappa_b + \epsilon$. Consequently,
\[
\liminf_{T \to \infty} \frac1T B_T(\infty) \leq F_b(\kappa_b + \epsilon).
\]
Since $\epsilon$ is arbitrary and we assumed that $F_b$ is continuous, this proves the result. The result for $\kappa_a$ is proved entirely similarly.
\end{proof}
Note that we only needed $F_b$ to be continuous at $\kappa_b$ and $F_a$ to be continuous at $\kappa_a$.

\subsection{Binning}
Before presenting the formal results in this section, we give the intuition. Consider a binned limit order book; recall that in an (ordinary) limit order book, a bid-ask pair may leave if they are in the same bin, even if the bid price is lower than the ask price. Now suppose we make bins larger. Intuitively, this should make it easier for bid-ask pairs to leave, so we expect to find fewer bids and asks in the system. The intuition is reversed for limit order books, where bid-ask pairs in the same bin are \emph{not} allowed to leave: there, making bins larger should leave more unfulfilled orders.

For two different partitions $\Pi$, $\tilde \Pi$ of $\BR$ into bins, we say that $\Pi$ refines $\tilde \Pi$ if every bin of $\tilde \Pi$ is the union of one or more bins of $\Pi$.

\begin{lem}\label{lm: bin monotonicity}
Let $\CL$ and $\tilde \CL$ be two ordinary binned limit order books with the same initial state and arrival process, and suppose that the binning partition $\Pi$ of $\CL$ refines the partition $\tilde \Pi$ of $\tilde \CL$. Then at all times $t$ and prices $p$ we have
\[
\tilde B_t(p) \leq B_t(p), ~~ \tilde A_t(p) \leq A_t(p).
\]
\end{lem}

\begin{proof}
The proof proceeds again by induction on the number of arrived orders. We show the inequality for $\tilde B_t \leq B_t$. Clearly it holds before any orders arrive.

Suppose at time $t$, an arrival of a bid at price $p$ occurs. In order to destroy the inequality $\tilde B \leq B$, we would need to have $p$ join the book in $\tilde \CL$ but leave in $\CL$, i.e. we must have
\[
\alpha_{t-} \lesssim p \prec \tilde \alpha_{t-}.
\]
Here, we mean that either $\alpha_{t-} < p$ or they occur in the same $\Pi$-bin, and $\tilde\alpha_{t-} > p$ and occurs in a different $\tilde\Pi$-bin. Since $\Pi$ is finer than $\tilde \Pi$, we must have $\tilde \alpha_{t-} > \alpha_{t-}$.

Furthermore, in order to destroy $\tilde B \leq B$ with the arrival of a single bid at price $p$, we must have had $\tilde B_{t-}(p) = B_{t-}(p)$ with equality. Note, however, that $B_{t-}(p) = B_{t-}(\infty)$, since the lowest ask in $\CL$ is in the same bin as $p$. Since at time $t-$ the inequality $\tilde B_{t-} \leq B_{t-}$ held, we have $\tilde B_{t-}(\infty) = B_{t-}(\infty)$.

Because bid-ask departures always occur in pairs, and the arrival processes were the same in $\tilde \CL$ and $\CL$, equality between total number of remaining bids implies equality for asks: $\tilde A_{t-}(-\infty) = A_{t-}(-\infty)$. Together with $\tilde A_{t-} \leq A_{t-}$, this implies $\tilde \alpha_{t-} \leq \alpha_{t-}$, and we've reached a contradiction. Thus, bid arrivals cannot destroy the inequality $\tilde B \leq B$.

Next, suppose that at time $t$ an arrival of an ask at price $p$ occurs. If this is to destroy $\tilde B \leq B$, then we must have
\[
\tilde \beta_{t-} \prec p \lesssim \beta_{t-},
\]
i.e. the ask leaves with the bid at $\beta_{t-}$ in $\CL$ but does not remove a bid in $\tilde \CL$. As before, these inequalities imply $\tilde \beta_{t-} < \beta_{t-}$.

Moreover, in order for the removal of the bid at $\beta_{t-}$ to destroy the inequality $\tilde B \leq B$, we must have $\tilde B_{t-}(\beta_{t-}) = B_{t-}(\beta_{t-})$. Now, by definition $B_{t-}(\beta_{t-}) = B_{t-}(\infty)$, and since $\tilde \beta_{t-} < \beta_{t-}$, also $\tilde B_{t-}(\beta_{t-}) = \tilde B_{t-}(\infty)$. We conclude $\tilde B_{t-}(\infty) = B_{t-}(\infty)$, which (since $\tilde B_{t-} \leq B_{t-}$) implies $\tilde \beta_{t-} \geq \beta_{t-}$, a contradiction. Thus, ask arrivals also cannot destroy the inequality $\tilde B \leq B$, and we are done.
\end{proof}

Entirely similarly, we can prove the corresponding statement for strict binned order books, in which the inequalities are reversed. We record the statement here for future reference.

\begin{lem}\label{lm: strict bin monotonicity}
Let $\CL$ and $\tilde \CL$ be two strict binned limit order books with the same initial state and arrival process, and suppose that the binning partition $\Pi$ of $\CL$ refines the partition $\tilde \Pi$ of $\tilde \CL$. Then at all times $t$ and prices $p$ we have
\[
\tilde B_t(p) \geq B_t(p), ~~ \tilde A_t(p) \geq A_t(p).
\]
\end{lem}

We obtain the following easy corollary.
\begin{cor}\label{coroll: ordinary and strict}
Consider three binning partitions $\Pi$, $\tilde \Pi$, $\hat \Pi$ where $\Pi$ refines both $\tilde \Pi$ and $\hat \Pi$. Let $\CL$ and $\tilde\CL$ be ordinary binned limit order books with bin partitions $\Pi$ and $\tilde\Pi$, and let $\hat\CL$ be a strict limit order book with bin partition $\hat\Pi$. Let the initial states and arrival processes be the same, and let $\kappa_i$, $i=a,b$ be defined for the books as in Theorem~\ref{thm:kappa}. Then
\[
\hat\kappa_b \geq \kappa_b \geq \tilde\kappa_b, \quad 1-\hat\kappa_a \geq 1-\kappa_a \geq 1-\tilde\kappa_a.
\]
\end{cor}

In the next section, we see that if the binning partitions are sufficiently fine, then $\hat \kappa_b$ and $\tilde \kappa_b$ are close to each other, which will allow us to compute the value of $\kappa_b$ for an ordinary unbinned limit order book using finer and finer binning partitions.

\section{Many-bin limit}
In this section, our goal is to show that we can reduce the analysis of the ordinary limit order book to the analysis of binned models. This is easier, because we are then reduced to a countable state-space Markov chain.

We begin with a bound on the effect of changing the arrival process on the value of $\kappa_b$. 
\begin{lem}
Consider two limit order books $\CL$ and $\tilde \CL$ with the same matching rule, but different arrival processes. Let $p_i, f_i$, $i=a,b$ denote the probability that an arriving order in $\CL$ is of type $i$ (bid or ask), and the density of the price of arriving orders. Let $\tilde p_i, \tilde f_i$, $i=a,b$ denote the corresponding quantities for $\tilde\CL$. Finally, let $\kappa_i, \tilde \kappa_i$, $i=a,b$ be given as in Theorem~\ref{thm:kappa}. Then
\[
\abs{\kappa_b - \tilde\kappa_b} \leq \abs{p_b - \tilde p_b} \left( \norm{f_b - \tilde f_b}_{TV} + \norm{f_a - \tilde f_a}_{TV} \right),
\]
and similarly for $\kappa_a$.
\end{lem}

\begin{proof}
We will use the characterization of $\kappa_b = \liminf_{T \to \infty} \frac1T B_T(\infty)$ given in Corollary~\ref{coroll: liminf}, together with Corollary~\ref{coroll:bounded perturbations}. We will set up the maximal coupling between the arrival processes for $\CL$ and $\tilde \CL$. First, we reparametrize time so that the arrival processes are Poisson, rate 1 in time, with densities $p_i f_i(p) dp \times dt$ for $\CL$, and $\tilde p_i \tilde f_i dp \times dt$ for $\tilde \CL$. Next, we construct the arrivals for $\CL$ and $\tilde\CL$ using six independent Poisson processes $P^{\text{com}}_b$, $P^1_b$, $P^2_b$, $P^{\text{com}}_a$, $P^1_a$, $P^2_a$ with densities
\begin{align*}
&P_b^{\text{com}}:~~(p_b \vee \tilde p_b) (f_b(p) \vee \tilde f_b(p)) dp\times dt, && P_b^{\text{com}}:~~(p_a \vee \tilde p_a) (f_a(p) \vee \tilde f_a(p)) dp\times dt\\
&P_b^1:~~(p_b - \tilde p_b)_+ (f_b(p) - \tilde f_b(p))_+ dp\times dt, && P_a^1:~~(p_a - \tilde p_a)_+ (f_a(p) - \tilde f_a(p))_+ dp\times dt\\
&P_b^2:~~(\tilde p_b - p_b)_+ (\tilde f_b(p) - f_b(p))_+ dp\times dt, && P_a^2:~~(\tilde p_a - p_a)_+ (\tilde f_a(p) - f_a(p))_+ dp\times dt.
\end{align*}

Now, let the arrival processes for $\CL$ be given by $P^{\text{com}}_i + P^1_i$, $i=a,b$, and let the arrival processes for $\tilde\CL$ be given by $P^{\text{com}}_i + P^2_i$. Note that the orders that arrive in one order book but not the other are simply the sum $P^1_i + P^2_i$. Aggregating these orders over time, the total difference in arrivals constitutes a Poisson process whose rate is bounded above by
\[
r = \abs{p_b - \tilde p_b} \left( \norm{f_b - \tilde f_b}_{TV} + \norm{f_a - \tilde f_a}_{TV} \right).
\]
Applying Corollary~\ref{coroll:bounded perturbations} together with the law of large numbers for Poisson processes, we see that this can change $\lim\int_{T \to \infty} \frac1T B_T(\infty)$ by at most $r$, hence the result.
\end{proof}

This observation will allow us to compute the threshold values $\kappa_i$, $i=a,b$ for the ordinary limit order book using ordinary binned limit order books. Note that we already know how to bound the threshold values between ordinary and strict limit order books by Corollary~\ref{coroll: ordinary and strict}. We now make the following observation.

Consider a strict binned order book with $N$ bins. Consider also an ordinary limit order book with $N+1$ bins, where bids arrive only in the leftmost $N$ bins according to the bid arrival process of the strict order book, and asks arrive only in the rightmost $N$ bins according to the ask arrival process of the strict order book. If we identify the leftmost $N$ bins with the strict limit order book, then the cumulative bid counts $B_t(p)$ will coincide in the two models; if we identify the rightmost $N$ bins instead, the cumulative ask counts $A_t(p)$ will coincide. If the partition into bins was quite fine, we expect these one-bin alterations should not make a large difference.

We now formalize this intuition.

\begin{lem}\label{lm: strict and ordinary}
Let $\CL^N$ be a sequence of ordinary, and $\tilde \CL^N$ a sequence of strict binned limit order books. Assume that the arrival process is the same for all of the limit order books, and the distribution of prices of arriving bids and asks is continuous and supported on $[0,1]$. Let the binning partition for $\tilde \CL^N$ have the following properties: (i) none of the bins receive more than $1/N$ of all of the bids, (ii) none of the bins receive more than $1/N$ of all of the asks; (iii) none of the bins have width more than $1/N$. (E.g., place boundaries between bins at $\frac{i}{N}$, $F_b^{-1}(\frac{i}{N})$ and $F_a^{-1}(\frac{i}{N})$ for $i=1,\dotsc,N$.) Let the binning partition for $\CL^N$ be the same as the binning partition of $\tilde \CL^{N+1}$.

Define the sequences $\kappa^N_i$ and $\tilde\kappa^N_i$, $i=a,b$ as in Theorem~\ref{thm:kappa}. Then, as $N \to \infty$,
\[
\abs{\kappa^N_b - \tilde \kappa^N_b} \to 0, \quad \abs{\kappa^N_a - \tilde \kappa^N_a} \to 0.
\]
\end{lem}

\begin{proof}
By the discussion above the statement of the lemma, we know that the value of $\kappa_b$ is the same in a strict limit order book with $N$ bins and a non-strict limit order book with $N+1$ bins and a slightly modified arrival process: we ignore at most $1/N$ of each of the bid and ask arrivals, and shift the remaining arrivals by one bin. To finish the proof, it remains to note that the total variation distance between the original and modified arrival distributions converges to 0, and the probability that an arriving offer is a bid in the modified arrival process converges to the original probability.
\end{proof}

Of course, the same observations apply to other features of the limit order book under a similar scaling.

\section{A limit along a subsequence}
We now turn to examining the long-term distribution of the locations $\beta_t$ and $\alpha_t$. In all of the analysis in this section, we parametrize time so that bids and asks both arrive as Poisson processes of rate 1 in time; in particular, we assume that arriving orders are equally likely to be bids and asks.

We first prove the following slightly stronger description of $\kappa_b$ and $\kappa_a$.
\begin{lem}\label{lm: Tn}
Almost surely there exists a sequence of times $T_n \to \infty$ along which
\begin{enumerate}
\item\label{lm1} $\frac{1}{T_n} B_{T_n}(\kappa_b) \to F_b(\kappa_b)$, $\frac{1}{T_n} A_{T_n}(\kappa_a) \to 1-F_a(\kappa_a)$;
\item\label{lm2} $\frac{1}{T_n} B_{T_n}(\infty) \to F_b(\kappa_b)$, $\frac{1}{T_n} A_{T_n}(-\infty) \to 1-F_a(\kappa_a)$.
\end{enumerate}
\end{lem}

\begin{proof}
Recall that for any $\epsilon > 0$ there exists a sequence of times $T^\epsilon_n \to \infty$ along which $\beta_t < \kappa_b + \epsilon$, and hence $B_{T^\epsilon_n}(\infty) = B_{T^\epsilon_n}(\kappa_b + \epsilon)$. Since each of $\beta_t < \kappa_b - \epsilon$ and $\alpha_t > \kappa_a + \epsilon$ occur only finitely many times, we may without loss of generality assume that these events do not occur after $T^\epsilon_1$. Consequently, after $T^\epsilon_1$, all bids arriving below $\kappa_b - \epsilon$ remain, and all asks arriving above $\kappa_b + \epsilon$ remain.

We now apply the law of large numbers to the bid and ask arrivals. Note that $B_t(x) - B_t(y)$ is always bounded above by the total number of bid arrivals between $x$ and $y$, and similarly for asks. We conclude that, by shifting indices on the sequence $T^\epsilon_n$, we can arrange the following:
\begin{enumerate}
\item $B_{T^\epsilon_n}(\infty) - B_{T^\epsilon_n}(\kappa_b) \leq (F_b(\kappa_b + \epsilon) - F_b(\kappa_b) + \epsilon) T^\epsilon_n$;
\item $(F_b(\kappa_b - \epsilon) - \epsilon)T^\epsilon_n \leq B_{T^\epsilon_n}(\kappa_b) \leq (F_b(\kappa_b) + \epsilon) T^\epsilon_n$;
\item $(1 - F_a(\kappa_a + \epsilon) - \epsilon)T^\epsilon_n \leq A_{T^\epsilon_n}(\kappa_a) \leq (1 - F_a(\kappa_a) + \epsilon) T^\epsilon_n$.
Moreover, the total numbers of bids and asks in the system differs by $o(T^\epsilon_n)$, and in fact, $O(\sqrt{T^\epsilon_n})$. Indeed, the difference in the number of arrivals of bids and asks is clearly the magnitude of a symmetric random walk; and they always depart in pairs. Recalling that $F_b(\kappa_b) + F_a(\kappa_a) = 1$, the above imply (decreasing $\epsilon$ if necessary) the fourth condition,
\item $A_{T^\epsilon_n}(-\infty) - A_{T^\epsilon_n}(\kappa_a) \leq (F_b(\kappa_b + \epsilon) - F_b(\kappa_b) + \epsilon) T^\epsilon_n$.
\end{enumerate}

We now pick a sequence of $\epsilon_n \to 0$, and take the diagonal subsequence of times:
\[
T_1 = T^{\epsilon_1}_1, ~~ T_{k+1} = \min_x \{T^{\epsilon_{k+1}}_x:~ T^{\epsilon_{k+1}}_x > T_k\}.
\]
\end{proof}

In what follows, we will be analyzing ordinary binned limit order books with finitely many bins. Define
\[
\pi^b_t(k) = \frac{1}{t} \int_0^{t} \one\{\beta_u \in \text{ bin $k$}\} du, \quad
\pi^a_t(k) = \frac{1}{t} \int_0^{t} \one\{\alpha_u \in \text{ bin $k$}\} du.
\]
Here, we say that if $\beta_t = -\infty$ then $\beta_t$ does not belong to any bin; similarly, if $\alpha_t = +\infty$ then $\alpha_t$ does not belong to any bin.

Let $\pi^b(\cdot)$ and $\pi^a(\cdot)$ be any limit point of $\pi^b_{T_n}(\cdot)$ and $\pi^a_{T_n}(\cdot)$, where $T_n$ is the sequence identified in Lemma~\ref{lm: Tn}. Our goal will be to analyze $\pi^b$ and $\pi^a$. We obtain the following characterization.

\begin{lem}\label{lm: pi eqns}
Let $k_b, k_a$ be the bins containing $\kappa_b$ and $\kappa_a$ respectively. Let $N$ be the total number of bins. Then $\pi^b$ and $\pi^a$ satisfy the following equations or inequalities:
\begin{subequations}\label{triv}
\begin{equation}\label{sum < 1}
\sum_{k=1}^N \pi^b(k) \leq 1; \quad \sum_{k=1}^N \pi^a(k) \leq 1.
\end{equation}
\begin{equation}
\pi^a(k) = \pi^b(k) = 0, ~~ k < k_b \text{ or } k > k_a.
\end{equation}
\end{subequations}
Equality holds in \eqref{sum < 1} if $\kappa_b > -\infty$ and $\kappa_a < +\infty$.

Let $a(k)$, $b(k)$ be the probability that an arriving order falls into bin $k$. That is, $a(k) = \int_{\text{bin $k$}} f_a(p) dp$, and similarly for bids. Then
\begin{subequations}\label{lln}
\begin{multline}\label{bid arrival-departure}
(1 - \sum_{l \leq k} \pi^a(l)) b(k) - \pi^b(k) (\sum_{l \leq k} a(l)) =\\
\begin{cases}
F_b(\kappa_b) - F_b(\inf\{x: x \in \text{ bin $k$}\}), & k = k_b\\
0, & k_b < k < k_a.
\end{cases}
\end{multline}
\begin{multline}\label{ask arrival-departure}
(1 - \sum_{l \geq k} \pi^b(l)) a(k) - \pi^a(k) (\sum_{l \geq k} b(l)) =\\
\begin{cases}
F_a(\sup\{x: x \in \text{ bin $k$}\}) - F_a(\kappa_a), & k = k_a\\
0, & k_b < k < k_a.
\end{cases}
\end{multline}
\end{subequations}
\end{lem}

\begin{proof}
Equations~\eqref{triv} follow from the definition of $\kappa_a$ and $\kappa_b$.

Equations~\eqref{lln} express the equation
\[
\text{order arrivals} - \text{order departures} = \text{unfulfilled orders}.
\]
The limiting number of unfulfilled orders in a given bin is given by Lemma~\ref{lm: Tn}. Let us show that the left-hand side represents the number of arrivals minus the number of departures.

Bids arrive into bin $k$ if the lowest ask is in some bin $a_t > k$, and then they arrive at rate $b(k)$. Formally, whenever the bin $a_t$ containing $\alpha_t$ satisfies $a_t > k$, the conditional probability that the next bid arrival will be into bin $k$ is $b(k)$; if $a_t \leq k$, the conditional probability is 0. Thus, conditional on the amount of time that $a_t > k$, the number of bid arrivals into bin $k$ is binomial with success probability $b(k)$. Since binomial random variables concentrate on their mean, we obtain the law of large numbers scaling above for the bid arrivals.

Similarly, bids depart from bin $k$ if $\beta_t$ is in bin $k$, and an ask arrives into some bin $l \leq k$. Therefore, conditional on the number of times that $\beta_t$ is in bin $k$, the number of bid departures is binomial with success probability $\sum_{l \leq k} a(l)$.
\end{proof}

The above is almost enough to determine the distributions $\pi^b$ and $\pi^a$, except for the inequality in \eqref{sum < 1}. Our next goal will be to find some sufficient conditions to conclude $\kappa_b > -\infty$ and $\kappa_a < \infty$.

\begin{lem}\label{lm:finite kappa}
Suppose that the binning and arrival price distributions for the limit order book are such that there exist prices $x < y$ with the following properties:
\begin{enumerate}
\item $0 < F_b(x) < F_b(y) < 1$, $0 < F_a(x) < F_b(y) < 1$;
\item The bin partition refines the partition $(-\infty,x] \cup (x,y] \cup (y,\infty)$;
\item $F_b(x) = 1-F_a(y)$ and $F_b(y) = 1-F_a(x)$.
\end{enumerate}
Then $\kappa_b$ and $\kappa_a$ satisfy
\[
1-F_a(\kappa_a) = F_b(\kappa_b) \geq \frac{2X(1-X)-(1-Y)}{(1-X)+(Y-X)},
\]
where $X = F_b(x)$ and $Y = F_b(y)$.

In particular, if we may choose $Y=1-X$, then $F_b(\kappa_b) > 0$ and $F_a(\kappa_a) < \infty$.
\end{lem}

\begin{proof}
We apply Lemma~\ref{lm: pi eqns} to the 3-bin partition appearing in the statement. By Lemma~\ref{lm: bin monotonicity}, this provides a lower bound on $\kappa_b$, and an upper bound on $\kappa_a$, for the original problem. The final condition implies that in the 3-bin partition, the probability of a bid arrival in bin $k$ is equal to $X$, $Y-X$, and $1-Y$ for the three bins, and the probability of ask arrival in bin $k$ is equal to $1-Y$, $Y-X$, and $X$. 

There will be six equations in \eqref{lln}. Note that in this scenario, $k_b = 4-k_a$, and they cannot both be equal to 2 since $\alpha_t$ and $\beta_t$ cannot be in the same bin.

Let us add together the following three pairs of equations: \eqref{bid arrival-departure} for bin 1 and \eqref{ask arrival-departure} for bin 3, \eqref{bid arrival-departure} for bin 2 and \eqref{ask arrival-departure} for bin 2, and finally \eqref{bid arrival-departure} for bin 3 and \eqref{ask arrival-departure} for bin 1. We obtain
\[
\begin{pmatrix}
X & 0 & 1-Y\\
Y-X & Y-X + 1-X & 0\\
1-Y+1 & 1-Y & 1-Y
\end{pmatrix}
\begin{pmatrix}
\pi^a(1) + \pi^b(3)\\
\pi^a(2) + \pi^b(2)\\
\pi^a(3) + \pi^b(1)
\end{pmatrix}
=2
\begin{pmatrix}
X - F_b(\kappa_b)\\
(Y-X)\\
(1-Y)
\end{pmatrix}.
\]
Premultiplying by $(1-X+Y-X, 1-Y, (1-Y)-2X(1-X))$ and observing that $\sum \pi^a(k) + \pi^b(k) \leq 2$, we obtain
\[
F_b(\kappa_b) \geq \frac{2X(1-X)-(1-Y)}{(1-X)+(Y-X)},
\]
as required. The final assertion follows because for $Y=1-X$ with $X < 1/2 < Y$, the numerator is strictly positive. (Note the denominator is always positive.)
\end{proof}

We now have obtained sufficient conditions for Lemma~\ref{lm: pi eqns} to provide us with the full description of the (discrete) distributions $\pi^a$ and $\pi^b$, \emph{if we knew the value of $\kappa_b$}.

\section{Proof of Theorem~\ref{thm:main}}\label{section:proof of main}
Let us summarize what we know so far. We are interested in finding the value of $\kappa_b$ in an ordinary (unbinned) limit order book. We know (Lemma~\ref{lm: strict and ordinary}) that it can be obtained by considering limit order books with smaller and smaller bins. Whenever we have a finite number of bins, Lemma~\ref{lm: pi eqns} tells us how to find $\pi^b$ and $\pi^a$, a pair of distributions supported on the bins between $\kappa_b$ and $\kappa_a$. We expect $\pi^b$ and $\pi^a$ to be the steady-state distribution of the (bin containing the) rightmost bid and of the (bin containing the) leftmost ask respectively, although we have only shown that it is the limiting distribution along a special sequence of times.

We now observe the following. Suppose we knew the values of $\kappa_b$ and $\kappa_a$. Then Lemma~\ref{lm: pi eqns} would allow us to compute $\pi^b$ and $\pi^a$, and we could rediscover $\kappa_b$ and $\kappa_a$ as the boundaries of their support. We will now see that the requirements that (i) Lemma~\ref{lm: pi eqns} should hold and (ii) $(\kappa_b, \kappa_a)$ should be the support of the resulting distributions are enough to determine $\kappa_b$ and $\kappa_a$.

\begin{proof}[Proof of Theorem~\ref{thm:main}]
Reparametrize coordinates so that all arrivals happen on $[0,1]$.

We consider a sequence of ordinary binned limit order books, $\CL^N$, which differ in the binning partitions $\Pi^N$ that they use. We require that $\Pi^N$ refines $\Pi^{N-1}$, and that each bin of $\Pi^N$ has width $\leq \frac1N$, and also $a(k), b(k) \leq \frac1N$. Here, $a(k)$ is the probability that an arriving ask enters bin $k$, and similarly for $b(k)$ and bids.

Suppose that in such a limit order book we do not know $\kappa_b$ and $\kappa_a$, but do know the bins $k_b = k_b(N)$ and $k_a = k_a(N)$ into which they fall. Then \eqref{lln} gives equations for $\pi^a(k)$ and $\pi^b(k)$ for all $k \neq k_a, k_b$, plus the following inequalities:
\begin{enumerate}
\item\label{1st ineq} $0 \leq \pi^b(k) \leq b(k)$, $0 \leq \pi^a(k) \leq a(k)$ for any $k$;
\item\label{2nd ineq} $0 \leq \pi^b(k_a) \leq \frac{\pi^a(k_a)}{F_a(\kappa_a)} b(k_a)$, $0 \leq \pi^a(k_b) \leq \frac{\pi^b(k_b)}{1-F_b(\kappa_b)}a(k_b)$.
\end{enumerate}
Now consider taking $N \to \infty$. By \eqref{1st ineq}, $\pi^b(k_b)$ and $\pi^a(k_a) \to 0$. Then by \eqref{2nd ineq},
\begin{equation}\label{eqn:right limit}
\frac{\pi^b(k_a)}{b(k_a)} \to 0, ~~ \frac{\pi^a(k_b)}{a(k_b)} \to 0.
\end{equation}
Note that these are the Radon-Nikodym derivatives of $\pi^b$ with respect to the bid arrival distribution $f_b$ (discretized to bins), and of $\pi^a$ with respect to $f_a$.

Consider now $\pi^b(k_b + 1)$. For it we have
\[
\pi^b(k_b+1) \left(\sum_{k \leq k_b+1} a(k)\right) = \left(\sum_{k > k_b+1} \pi^a(k)\right) b(k_b+1).
\]
By considerations similar to \eqref{1st ineq}, $\pi^a(k) \to 0$ uniformly as $N \to \infty$ for any single bin $k$. Consequently, as $N \to \infty$, $\sum_{k > k_b(N)+1} \pi^a(k) = 1 - \pi^a(k_b(N)) - \pi^a(k_b(N)+1) \to 1$, and we obtain
\begin{equation}\label{eqn:left limit 1}
\frac{\pi^b(k_b(N)+1)}{b(k_b(N)+1)} \to \left(\sum_{k \leq k_b+1} a(k)\right)^{-1}
\end{equation}
and similarly for asks.

An identical calculation yields $\frac{\pi^b(k_b(N)+2)}{b(k_b(N)+2)} \to \left(\sum_{k \leq k_b+2} a(k)\right)^{-1}$, from which
\begin{equation}\label{eqn:left limit 2}
\frac{\pi^b(k_b(N)+1)}{b(k_b(N)+1)} - \frac{\pi^b(k_b(N)+2)}{b(k_b(N)+2)} \to \frac{a(k_b(N)+2)}{\left(\sum_{k \leq k_b(N)+1} a(k)\right)\left(\sum_{k \leq k_b(N)+2} a(k)\right)},
\end{equation}
and similarly for asks. Note that we could have the summation running to $k_b(N)$ in the denominator, since $a(k) \to 0$ uniformly for any single bin $k$.

As $N \to \infty$, the derivatives $\frac{\pi^b(k)}{b(k)}, k=1,\dotsc,N$ and $\frac{\pi^a(k)}{a(k)}, k=1,\dotsc,N$ are bounded between 0 and $\frac{1}{F_a(\kappa_b)}$, resp. $\frac{1}{F_b(\kappa_a)}$, and hence converge along some subsequence. Consider any such pair of subsequential limits (we will shortly see that it is unique), $\varpi^b$ and $\varpi^a$. We find
\begin{subequations}
\[
0 \leq \varpi^b(x), \varpi^b(x);
\]
\[
\varpi^b(x) = \varpi^a(x) = 0, ~~ x < \kappa_b \text{ or } x > \kappa_a;
\]
\[
\int_{\kappa_b}^{\kappa_a} \varpi^b(x) f_b(x) dx = \int_{\kappa_b}^{\kappa_a} \varpi^a(x) f_a(x) dx = 1;
\]
\begin{align*}
&\varpi^b(x) F_a(x) = \left(\int_x^{\kappa_a} \varpi^a(y) f_a(y) dy\right), ~~ \kappa_b < x < \kappa_a\\
&\varpi^a(x) (1-F_b(x)) = \left(\int_{\kappa_b}^x \varpi^b(y) f_b(y) dy\right), ~~ \kappa_b < x < \kappa_a;
\end{align*}
\begin{align*}
\varpi^b(\kappa_b) = \frac{1}{F_a(\kappa_b)},\qquad &\frac{d}{dx}\varpi^b(x)\vert_{x=\kappa_b} = -\frac{f_a(\kappa_b)}{F_a(\kappa_b)^2},\\
\varpi^a(\kappa_a) = \frac{1}{1-F_b(\kappa_a)},\qquad &\frac{d}{dx}\varpi^a(x)\vert_{x=\kappa_a} = \frac{f_b(\kappa_a)}{(1-F_b(\kappa_a))^2};
\end{align*}
\[
\varpi^b(x) \to 0, ~~ x \uparrow \kappa_a; \qquad \varpi^a(x) \to 0, ~~ x \downarrow \kappa_b.
\]
\end{subequations}

It remains to observe that the pair of integral equations can be converted into a pair of differential equations for $\varpi^b(x)$. Indeed,
\[
\left(F_a(x)\varpi^b(x)\right)' = -\varpi^a(x) f_a(x) = -\frac{f_a(x)}{1-F_b(x)} \int_{\kappa_b}^x \varpi^b(y) f_b(y) dy,
\]
and hence
\[
\left(-\frac{f_a(x)}{1-F_b(x)} \left(F_a(x) \varpi^b(x) \right)' \right)' = \varpi^b(x) f_b(x).
\]

For this second-order ODE, we have two initial conditions -- the values $\varpi^b(\kappa_b)$ and $\frac{d}{dx} \varpi^b(x)\vert_{x=\kappa_b}$. This allows us to find a solution for the ODE given $\kappa_b$, for each value of $\kappa_b$. Recall, however, that we have an additional constraint $\varpi^b(x) \to 0$ as $x \to \kappa_a$, where $\kappa_a$ and $\kappa_b$ are related via $F_b(\kappa_b) = 1-F_a(\kappa_a)$. It is not difficult to see that there can only be one value of $\kappa_b$ that is consistent with the ODE and the additional constraint of $\varpi$ vanishing at $\kappa_a = F_a^{-1}(1-F_b(\kappa_b))$.
\end{proof}

\subsection{Calculations for uniform distribution}\label{section:pictures}
For the case of the uniform distribution, $f_a = f_b = \one_{[0,1]}$, we can take the calculations a step further by solving the above differential equation. We obtain
\[
\left(-\frac{1}{1-x}\left(x \varpi^b(x)\right)'\right)' = \varpi^b(x), ~~ \varpi^b(\kappa_b) = \frac{1}{\kappa_b}, ~~ \frac{d}{dx}\varpi^b(x)\vert_{x=\kappa_b} = -\frac{1}{\kappa_b^2},
\]
which can be solved explicitly to give
\[
\varpi^b(x) = (1-\kappa) \left(\frac1x + \log\left(\frac{1-x}{x}\right) \right), ~~ x \in (\kappa,1-\kappa).
\]
The value of $\kappa=\kappa_b$ is given as follows. Let $w$ be the unique solution of $w e^w = e^{-1}$; then $\kappa = \frac{w}{w+1} \approx 0.217$.

Figure~\ref{fig:empirical and exact} compares the empirical distribution of the location of $\beta_t$ when we consider 100 bins, and the curve $\varpi^b$ given above. The close agreement between the two curves supports Conjecture~\ref{conj:recurrent}. Further support is given by Figure~\ref{fig:number of bids}, which shows the total number of bids in bins to the right of the threshold as a function of time. The plot of the maximal value observed up to time $t$ as a function of $t$ is seen not to grow linearly, supporting the conjecture.

\begin{figure}[tbhp]
\begin{center}
\includegraphics[width=2in]{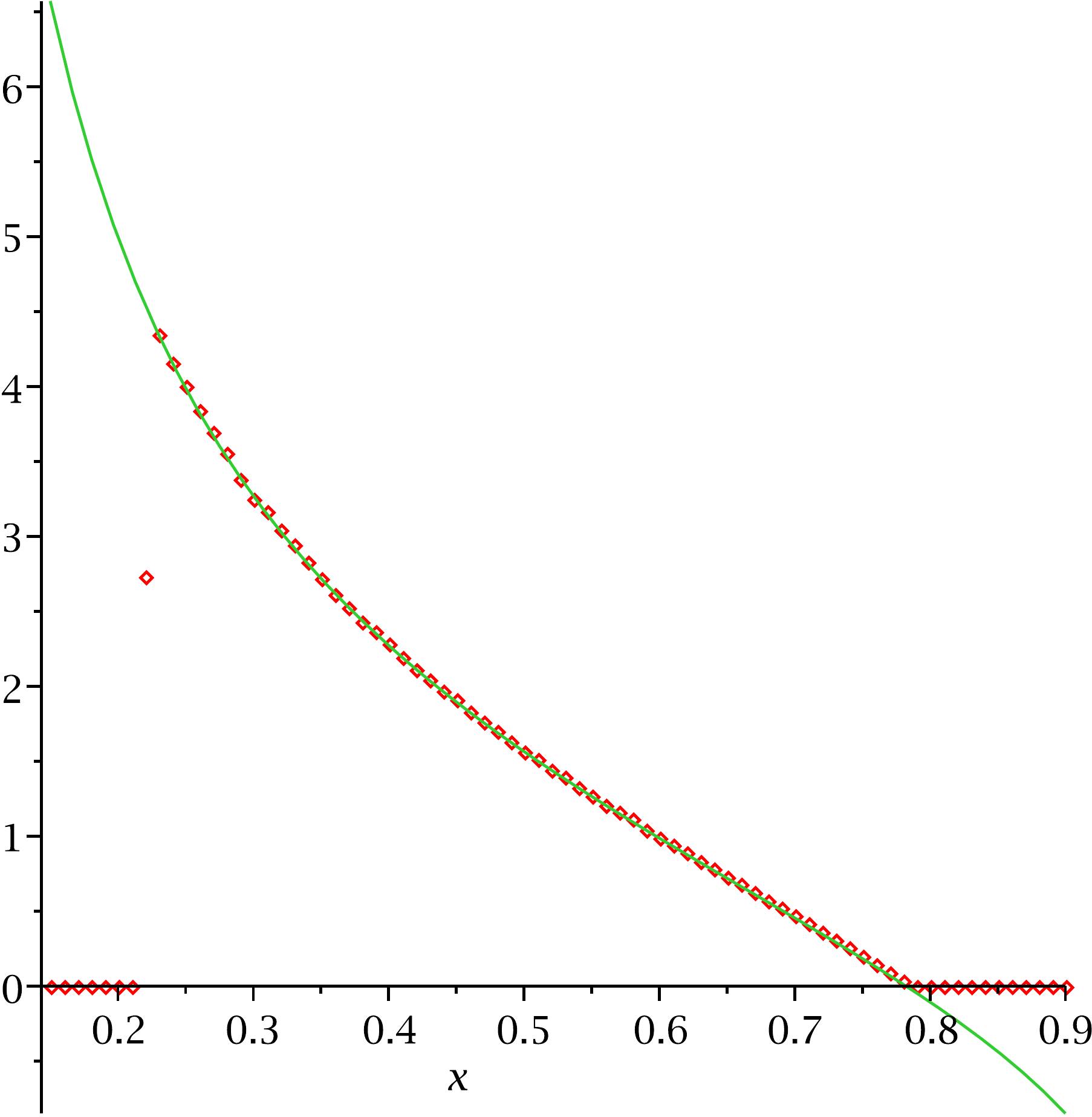}
\caption{Empirical distribution of the location of $\beta_t$ with uniform bid and ask arrivals and 100 bins, and the predicted density for the unbinned model.}
\label{fig:empirical and exact}
\end{center}
\end{figure}

\begin{figure}[tbhp]
\begin{center}
\includegraphics[width=2in]{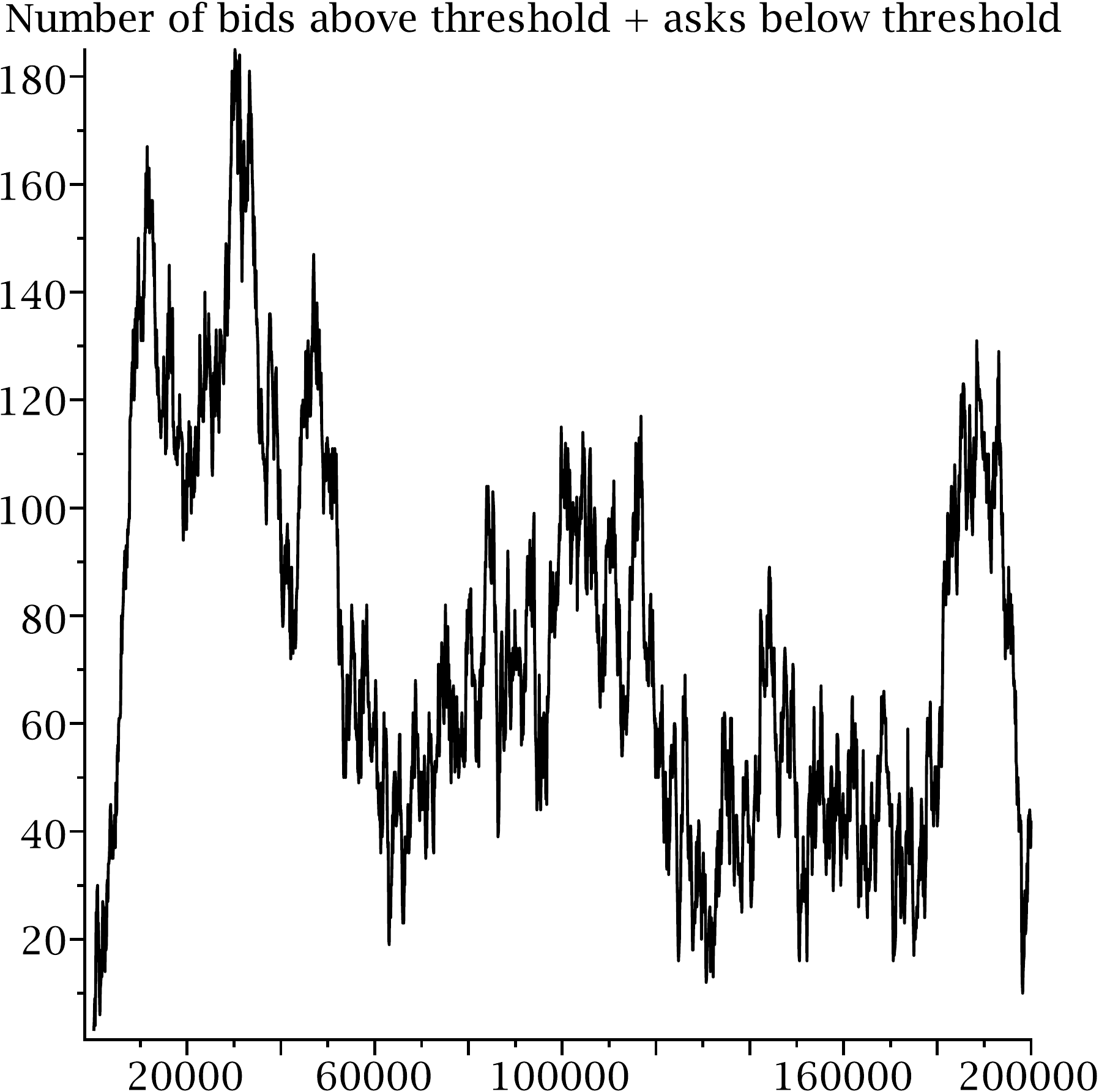} ~
\includegraphics[width=2in]{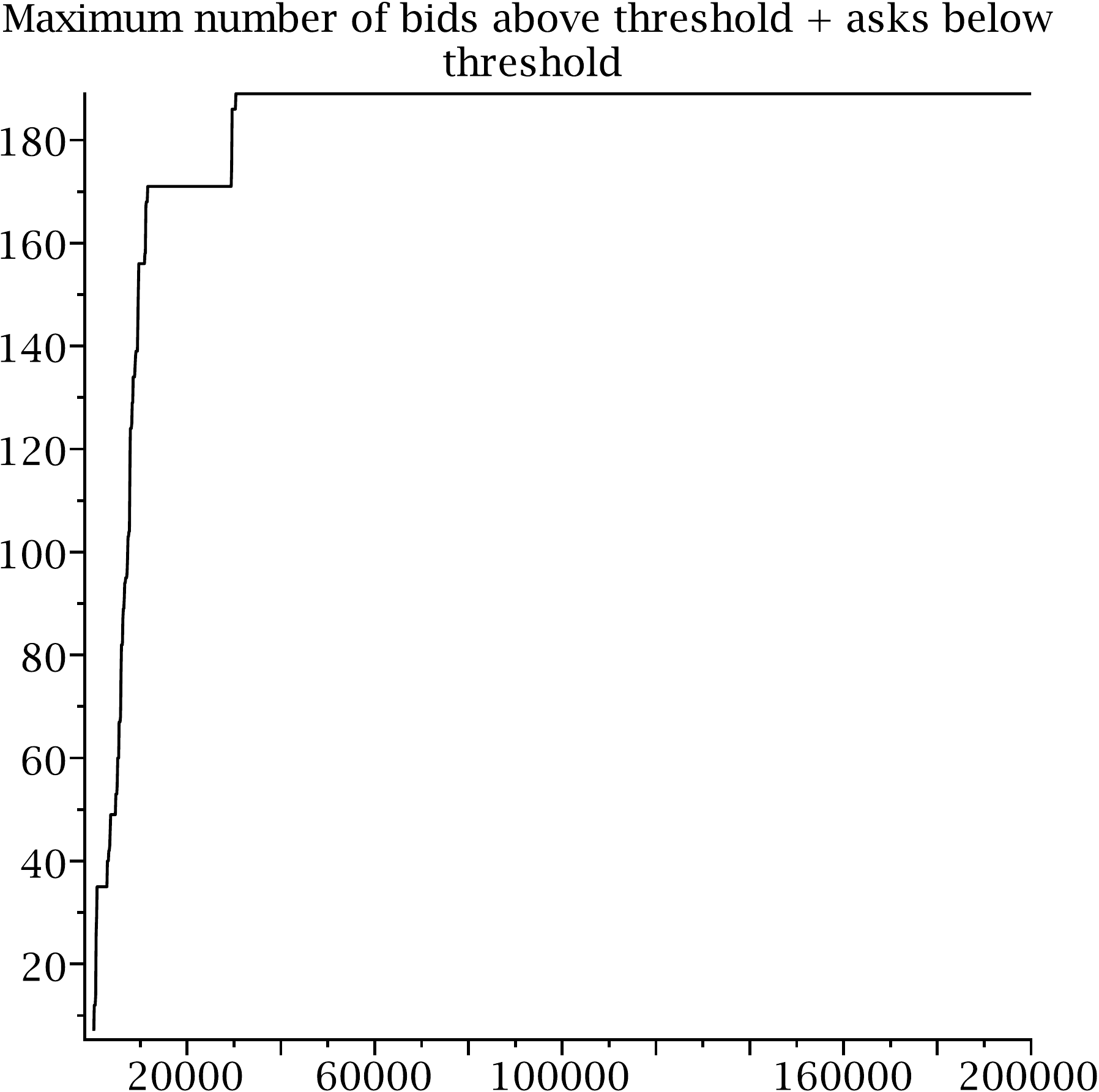}
\caption{Left: total number of bids above bin $k_b=21$ and asks below bin $k_a=79$, as a function of time. Right: running maximum of this quantity, as a function of time. The last jump of the running maximum occurs at arrival 30,309 of 200,000.}
\label{fig:number of bids}
\end{center}
\end{figure}

\section{Recurrence}\label{section:recurrence}
In this section, our goal is to prove results similar to Conjecture~\ref{conj:recurrent}. While we will not be able to derive recurrence when there is an infinite supply of bids and asks at $\kappa_b$ and $\kappa_a$, we will be able to derive it for smaller subintervals.

\begin{thm}
Let the arrivals of both bids and asks are Poisson of rate 1 in time, with densities $f_b$ and $f_a$ respectively in price. Suppose that there exist values $x$ and $y$ such that $F_b(y)<F_b(x)+F_a(x)$ and $F_a(y) < F_a(x) + (1-F_b(y))$. (For example, we may have $F_a(x) = F_b(x) > 1/3$ and $F_a(y) = F_b(y) < 2/3$.) Let $\CL$ be a limit order book whose initial state is such that there are infinitely many bid orders at $x$ and infinitely many ask orders at $y$. Letting the state of $\CL$ be described by the bids and asks in $(x, y)$, it is a positive Harris recurrent Markov chain.
\end{thm}

\begin{proof}
Consider the bids in $(x,y)$. They arrive at rate at most $b(3) = F_b(y)-F_b(x)$. Moreover, whenever there are any bids in bin 3, they depart at rate at least $F_a(x)$. Consequently, the number of bids in bin 3 is (stochastically) bounded above by a geometric random variable with parameter $\frac{F_b(y)-F_b(x)}{F_a(x)} < 1$. Similarly, the number of asks in bin 3 is bounded above by a geometric random variable with parameter $\frac{F_a(y)-F_a(x)}{1-F_b(y)} < 1$. This suffices to prove the claim.
\end{proof}
Note that in particular this result gives an upper bound on the value of $\kappa_b$ (and a lower bound on $\kappa_a$).

We can prove a slightly stronger result for binned limit order books.
\begin{thm}
Suppose arriving orders are equally likely to be bids and asks, and $f_a = f_b = \one_{[0,1]}$. Let $\epsilon > 0$, and consider the binned limit order book with 5 bins of sizes $1/5+\epsilon$, $1/5-\epsilon$, $1/5$, $1/5-\epsilon$, $1/5+\epsilon$ whose initial state has infinitely many bids in bin 1 and infinitely many asks in bin 5. This limit order book (considered on the middle three bins only) is positive recurrent.
\end{thm}

\begin{proof}
Let $X_t \in \BZ^3$ be the Markov chain describing their state. We let $\abs{X_t(i)}$ denote the number of orders in bin $i$, and its sign correspond to bids ($+$) or asks ($-$).

The evolution of the system depends on the bins containing the rightmost bid and of the leftmost ask; call these $b_t$ and $a_t$ respectively. There are 10 possible combinations, which we denote $+++$, $++-$, $+--$, $---$, $++0$, $+0-$, $0--$, $+00$, $00-$, and $000$. The signs should be thought of as the signs of $X_t(i)$, although we do not distinguish e.g. $0++$ from $+++$. Note that $000$ corresponds to the middle three bins being empty.

Consider the (vector) drift of $X$, that is, $\BE[X_{t+1} - X_t \vert X_t]$. As we mentioned, this depends only on the region $F$ to which $X_t$ belongs, where $F$ comes from the list of possible descriptions for the pair $(b(t), a(t))$. We will not be interested in the drift when $F=000$. The drifts are as follows:
\begin{align*}
&\Delta_{+++} = (1/5-\epsilon,1/5,-(4/5-\epsilon)), && \Delta_{---} = (4/5-\epsilon,-1/5,-(1/5-\epsilon)),\\
&\Delta_{++-} = (1/5-\epsilon,-3/5,2/5-\epsilon), && \Delta_{+--} = (-(2/5-\epsilon),3/5,-(1/5-\epsilon)),\\
&\Delta_{++0} = (1/5-\epsilon,-3/5,0), && \Delta_{0--} = (0,3/5,-(1/5-\epsilon)),\\
&\Delta_{+0-} = (-(2/5-\epsilon),0,2/5-\epsilon), && \Delta_{+00} = (-(2/5-\epsilon),0,0),\\
&\Delta_{00-} = (0,0,2/5-\epsilon).
\end{align*}

We will now show that $X$ is positive recurrent by constructing a Lyapunov function for it. Note that the jumps of $X$ are bounded by 1. We define $\CL(x)$ by
\[
\CL(x) = \min(\langle x, v_F \rangle), F \in \{++-, +--, ++\,0, +00, +0-, 0--, 00-\}.
\]
We will specify the vectors $v_F$ shortly. The set of possible values of $F$ is the set of relative positions of bid and ask above, except for $000$ (the origin).

The level sets of this Lyapunov function are polyhedra with outer normals $v_F$. The vectors $v_F$ will be picked so that a point $x$ on the face of the level polyhedron with outer normal $v_F$ belongs either to the orthant $F$, or (if $F$ is not an orthant, i.e. if it contains 0) to one of the orthants adjacent to $F$. In this way, it will be sufficient to construct $v_F$ so that $\langle \Delta_{\tilde F}, v_F \rangle < 0$ whenever $\tilde F$ agrees with $F$ at all the nonzero places of $F$. By compactness of the level sets, this will guarantee that
\[
\exists K > 0 \text{ such that } \BE[\CL(X_{t+1}) - \CL(X_t) \vert X_t] < -\epsilon < 0 \text{ whenever $\CL(X_t) > K$}.
\]
Together with the fact that the jumps of $\CL$ are clearly bounded above (because the jumps of $X$ are, and $\CL$ is Lipschitz), this gives the Foster-Lyapunov criterion for positive recurrence as used e.g. in \cite[Proposition 4.4]{Bramson}.

It can be checked that the choice
\begin{align*}
&v_{+++} \equiv (1,1,1), && v_{++-} \equiv (1,1,-1),\\
&v_{+--} \equiv (1,-1,-1), && v_{---} \equiv (-1,-1,-1),\\
&v_{++0} = v_{+00} \equiv \left(4/3,1,2/3\right), &&v_{+0-} \equiv \left(1,-4/5,-9/5\right),\\
&v_{0--} = v_{00-} \equiv (-2,-3,-4).
\end{align*}
satisfies all of these constraints. Figure~\ref{fig:lyapunov} shows the corresponding level set of $\CL$.

\begin{figure}[htbp]
\begin{center}
\includegraphics[width=2in]{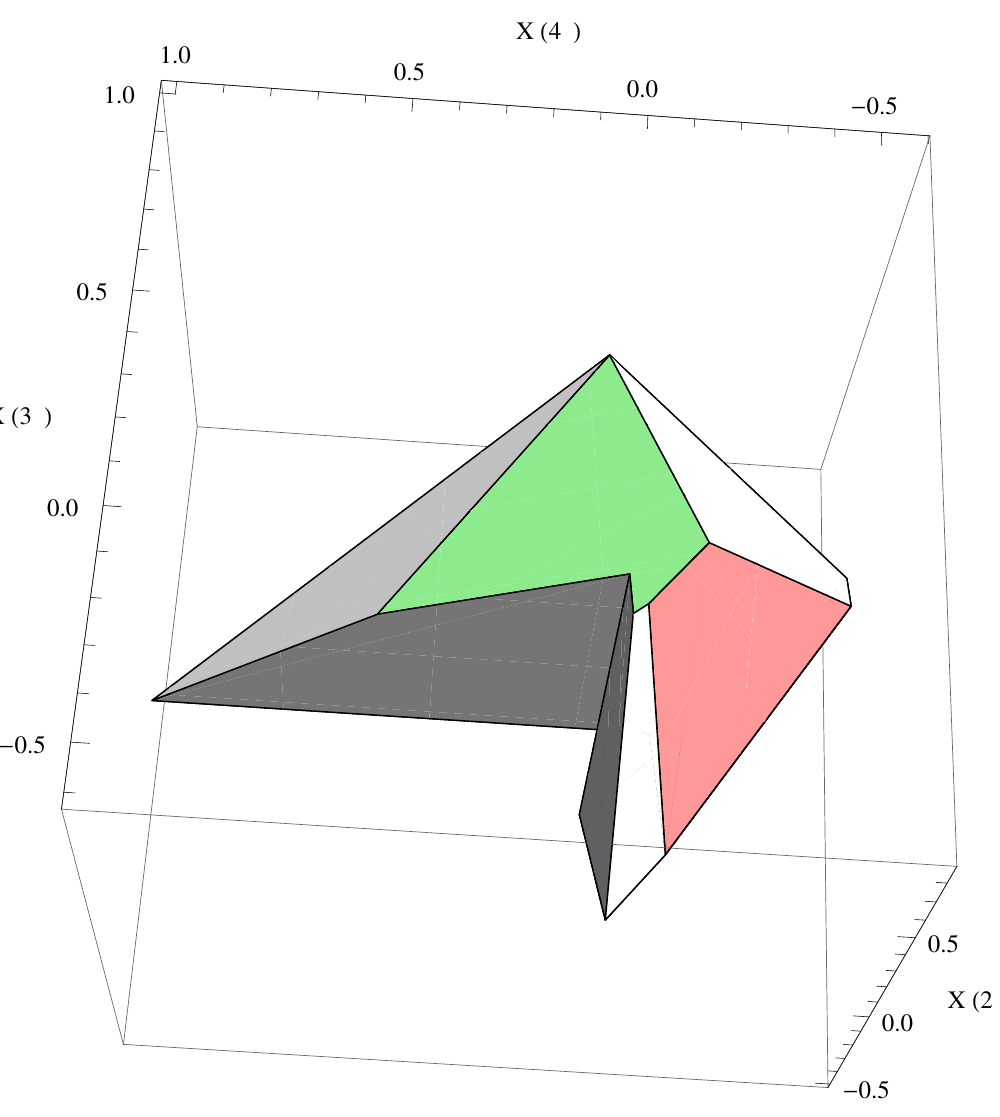} \quad
\includegraphics[width=2in]{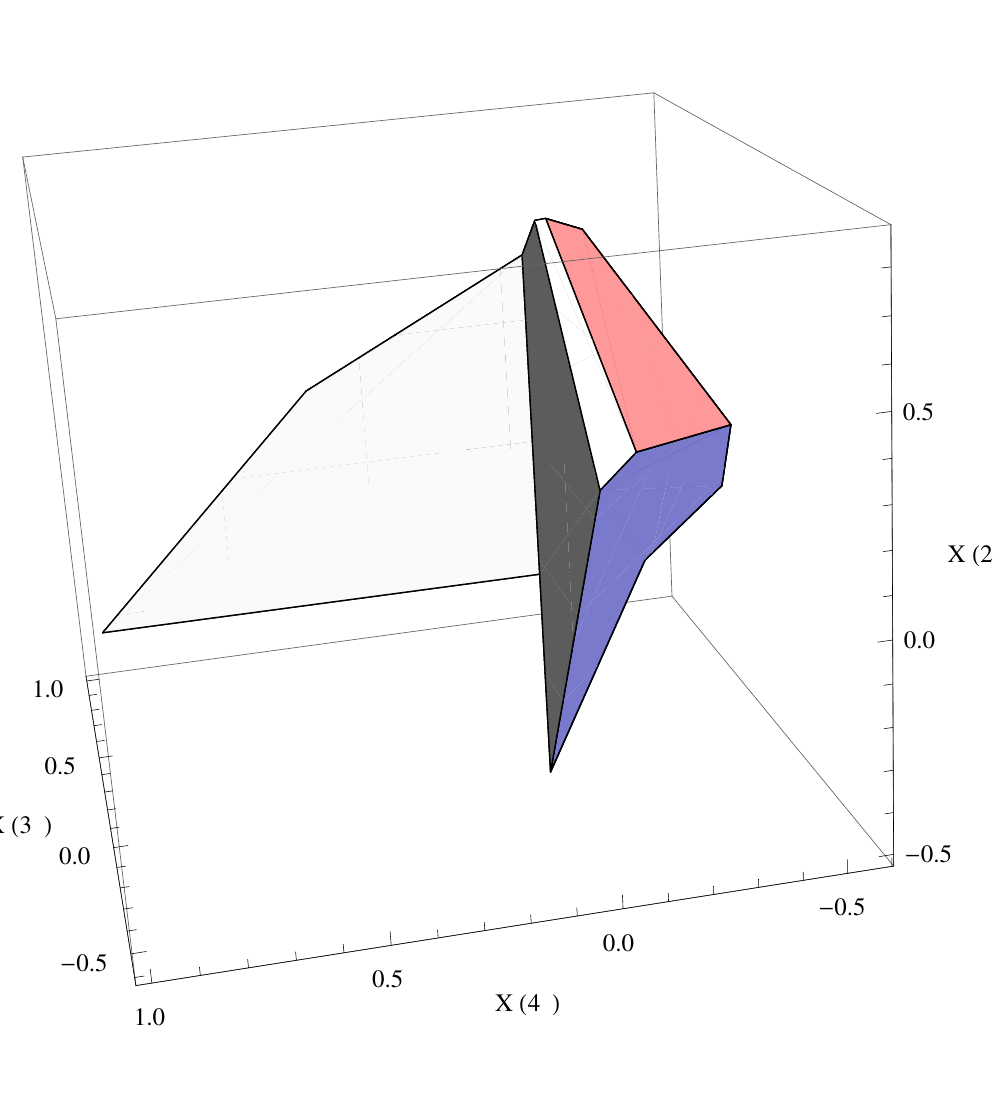}
\caption{Two views of the level set $P \equiv \{x: \CL(x)=1\}$. $P$ has 15 vertices $(0, 0, 0)$, $(0, 1, 0)$, $(0, 0, 1)$, $(\frac12, 0, \frac12)$, $(\frac{45}{58}, \frac{2}{29}, -\frac{9}{58})$, $(\frac67, -\frac17, 0)$, $(\frac{29}{34}, -\frac{2}{17}, -\frac{1}{34})$, $(\frac34, 0, 0)$, $(\frac{11}{50}, \frac{6}{25}, -\frac{27}{50})$, $(0, \frac37, -\frac47)$, $(\frac{11}{26}, -\frac{6}{13}, -\frac{3}{26})$, $(\frac25, -\frac35, 0)$, $(0, -\frac13, 0)$, $(-\frac12, 0, 0)$, $(0, 0, -\frac14)$, and 10 faces (defined as ordered sets of vertices, possibly not all with the same orientation) $\{4, 3, 2\}$, $\{5, 2, 10, 9\}$, $\{7, 6, 12, 11\}$, $\{1, 3, 4, 8\}$, $\{1, 8, 6, 12, 14\}$, $\{1, 3, 2, 10, 15\}$, $\{1, 15, 14\}$, $\{7, 5, 9, 11\}$, $\{2, 4, 8, 6, 7, 5\}$, $\{9, 10, 15, 14, 12, 11\}$.}
\label{fig:lyapunov}
\end{center}
\end{figure}
\end{proof}

This result does not directly translate into a statement about recurrence of an unbinned limit order book. However, it can be used to derive some bounds on $\kappa_b$ and $\kappa_a$ for an ordinary (unbinned) limit order book with uniform arrivals of bids and asks. Analysis similar to that of Lemma~\ref{lm: strict and ordinary} gives, by looking at strict limit order books with 4 bins, $\kappa_b < 1/4$ and $\kappa_a > 3/4$. We do not go into details here, since for this case we have already computed the value of $\kappa_b$ precisely in Section~\ref{section:pictures}.

One further indication that the positive Harris recurrence holds is obtained by plotting the empirical density of the joint location of the highest bid and the lowest ask. Figure~\ref{fig: joint} presents the plots obtained by simulation. The plots suggest that there is a limiting surface describing the joint density, although we have been unable to obtain an expression for it.

\begin{figure}[htbp]
\begin{center}
\includegraphics[width=3in]{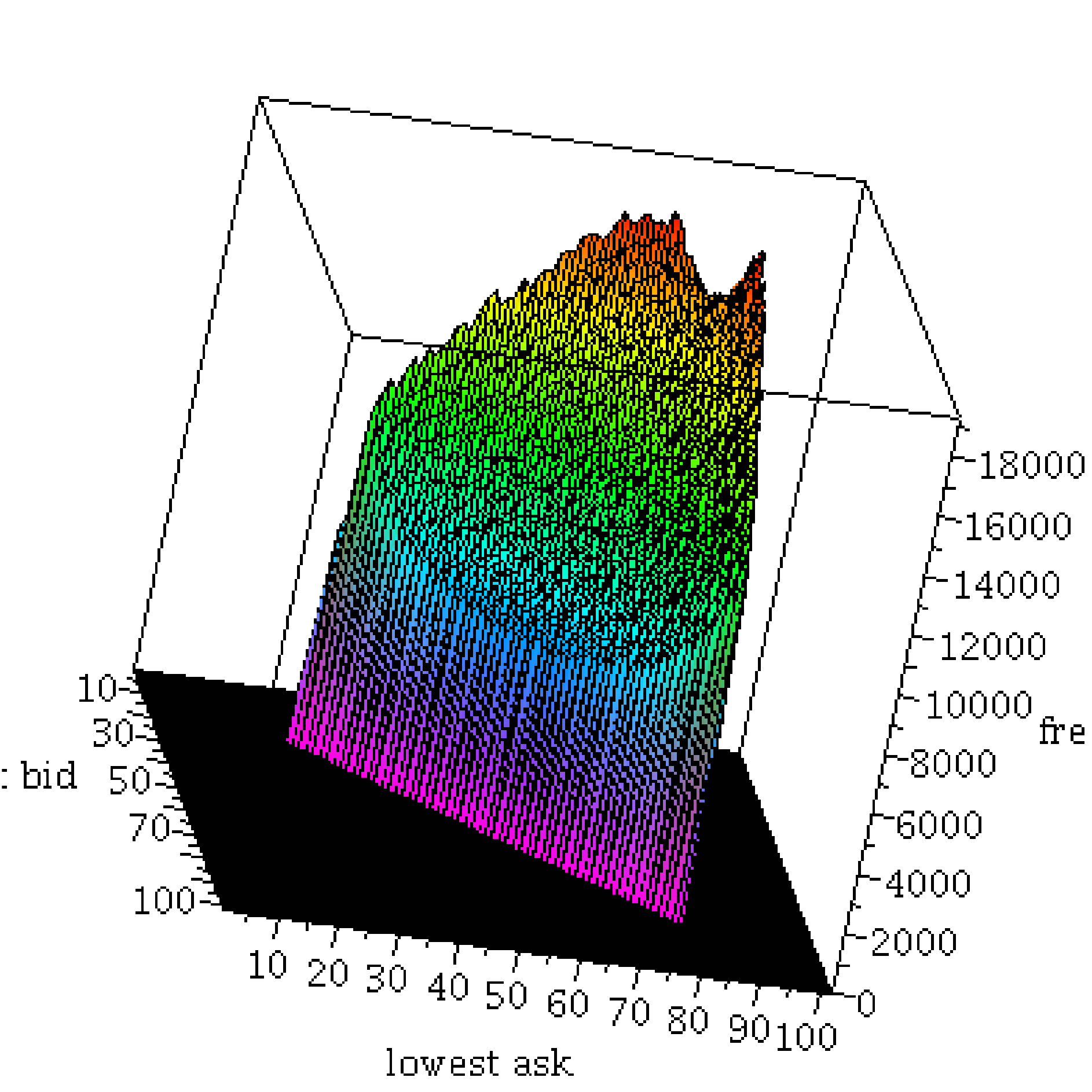} ~
\caption{Joint distribution of the rightmost bid and the leftmost ask.}
\label{fig: joint}
\end{center}
\end{figure}

\section{Discussion and future work}\label{section: future}
We have presented a model of a limit order book applicable on relatively short time scales, during which the price is relatively stationary, order arrivals can be modeled as having distributions independent of the state of the limit order book, and abandonments can be ignored (either as background noise or as orders that are never executed). We have obtained a phase transition, which suggests that at such a time orders in a narrow band around the price get executed (more and more rarely as we move away from the price), while orders farther away from the price will never be reached. We have determined this band in terms of the distribution of the prices of arriving orders. Modulo a conjectural result on recurrence, we have also determined how frequently orders at a given distance away from the optimal price get executed.

There are two major directions in which this work could be extended.

First, it would be desirable to understand finer features of the model. In particular, we would like to gain an understanding of the joint distribution of the highest bid and lowest ask, which would also allow us to understand the distribution of the bid-ask spread in this model. It would also be interesting to examine the limiting shape of the book. Note that the distribution of the rightmost bid gives an indication of the expected number of bids and asks at any given price, but we would like to understand the following: conditional on $\beta_t$ being in bin $b_t$ and $\alpha_t$ being in bin $a_t$, what is the number of bids in bin $b_t - k$? What is the number of asks in bin $a_t + k$? Can we reproduce some of the empirical results concerning these shapes?

The left-hand side of Figure~\ref{fig:numTopBids}, we present the number of bids in the bin containing the highest bid, as a function of time. The occasional spikes are not surprising, because we expect the highest bid to occasionally be in the threshold bin, which contains large numbers of orders. The right-hand side presents the average number of bids at and near the bid price, \emph{when the bid price is high}. We see that there often are no bid orders near the highest waiting bid, when the arrival distribution is uniform. This may, however, change if we consider different arrival distributions.

\begin{figure}[tbhp]
\begin{center}
\includegraphics[height=1.8in]{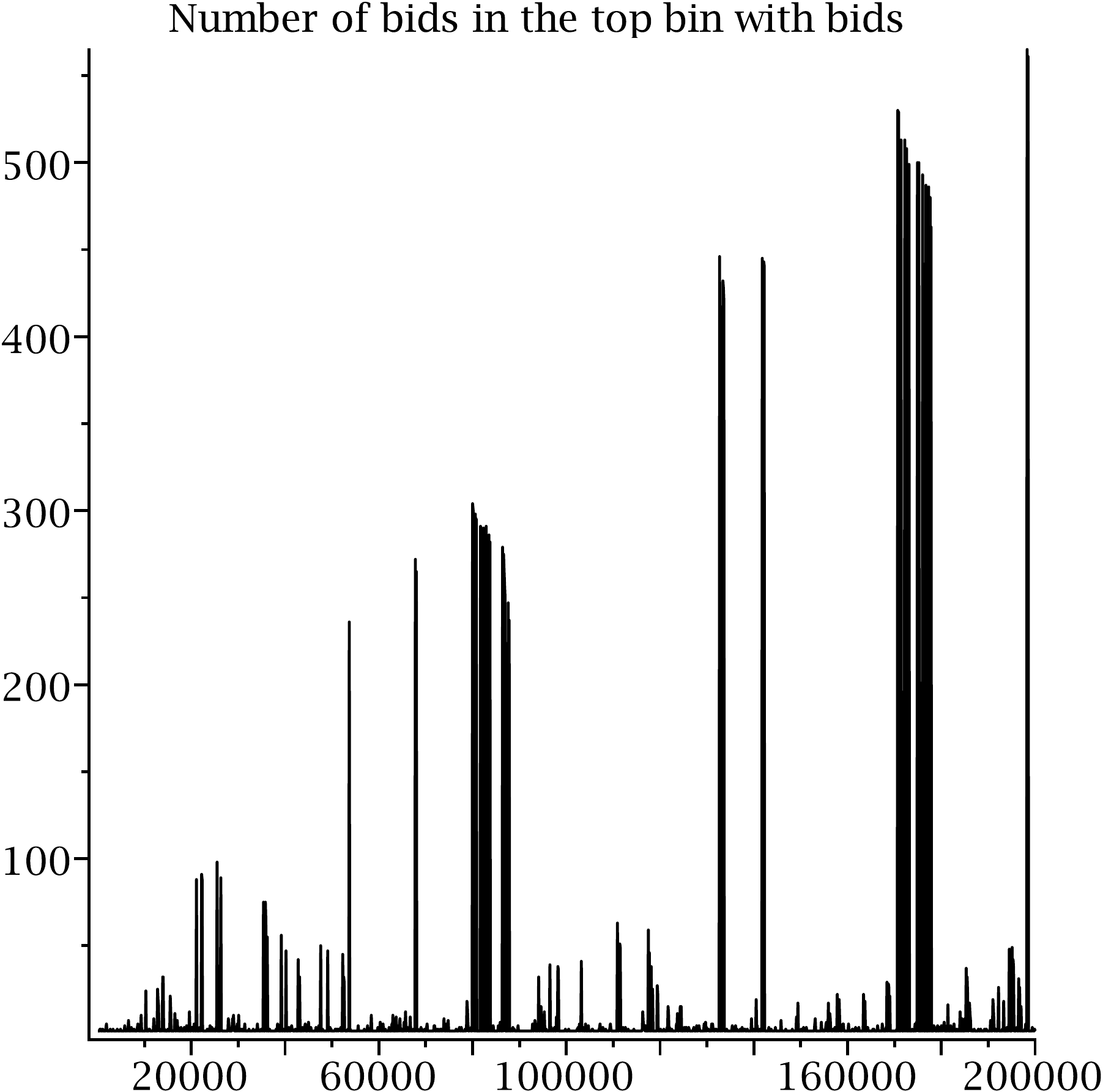}~
\includegraphics[height=1.8in]{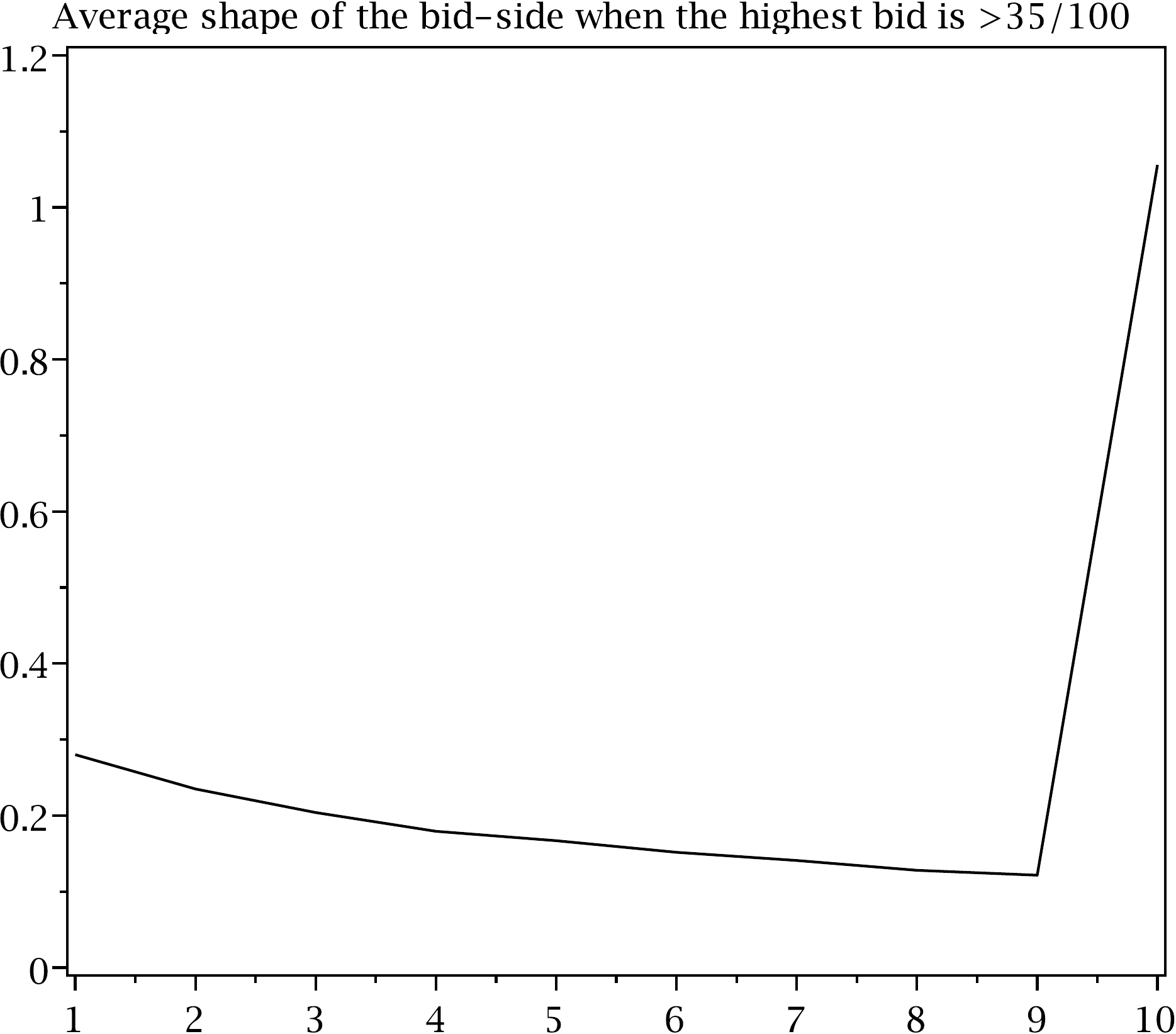}
\caption{The number of bids in the bin with the highest bid; the average shape of the bid side of the limit order book when the bid price is high.}
\label{fig:numTopBids}
\end{center}
\end{figure}

Second, it would be desirable to change the model so that it would more closely agree with actual limit order books. In particular, we would like to be able to incorporate order abandonment, non-unit-sized orders, and order distributions that depend on the price. Order abandonment fundamentally changes our analysis by removing the hard phase transition and getting rid of monotonicity properties; it is likely that it requires a different formulation of the model. However, different arrival distributions may be amenable to analysis once we have a better understanding of the shape of the limit order book around its price.

\acks
The author would like to thank Yuri Suhov and Frank Kelly for introducing her to limit order books and for many insightful conversations, and Vlada Limic and Florian Simatos for interesting conversations about this and related models. Special thanks to Daniel Whalen for help generating the images in Figure~\ref{fig:lyapunov}.

This material is based upon work supported by the National Science Foundation Graduate Research Fellowship and Award No. 1204311.

\bibliographystyle{apt}
\bibliography{LOB}

\end{document}